\documentclass[reqno]{amsart}
\usepackage[displaymath]{lineno}
\usepackage{amssymb,amsthm,amsmath,hyperref,bbm,color}
\usepackage{geometry}
\makeindex
\hypersetup{
	colorlinks = true,
	linkcolor = blue,
	citecolor = blue,
	pdfauthor = 'Jason Polak',
	pdfkeywords = '2da3e562039dc07140d88056c11480c1'
}
\usepackage[all]{xy}
\definecolor{orangetwo}{rgb}{0.9803922, 0.4196078, 0.3058824}
\definecolor{bluetwo}{rgb}{0.1843137, 0.3843137, 0.5176471}
\definecolor{freading}{rgb}{ 0.8980392,1.0000000,0.9058824}
\definecolor{gray}{rgb}{0.25,0.25,0.25}
\definecolor{dgray}{rgb}{0.15,0.15,0.15}
\definecolor{white}{rgb}{1,1,1}

\usepackage{fancyhdr}
\newtheoremstyle{nthm}
{3pt}
{3pt}
{\itshape}%body font
{0em}%indent amount
{\bfseries}%theorem head font
{.}%puncuation after theorem head
{.5em}%space after theorem head
{}%theoremheadspec who knows/??

\newtheoremstyle{ndef}
{3pt}
{3pt}
{}%body font
{0em}%indent amount
{\bfseries}%theorem head font
{.}%puncuation after theorem head
{.5em}%space after theorem head
{}%theo

\newtheoremstyle{nrem}
{3pt}
{3pt}
{}%body font
{0em}%indent amount
{\itshape}%theorem head font
{}%puncuation after theorem head
{.5em}%space after theorem head
{}%theor

\swapnumbers
\theoremstyle{nthm}
\newtheorem{thm}[subsubsection]{Theorem}
\newtheorem{lem}[subsubsection]{Lemma}
\newtheorem{prop}[subsubsection]{Proposition}

\theoremstyle{ndef}
\newtheorem{defn}[subsubsection]{Definition}
\theoremstyle{nrem}

\newtheorem{rem}[subsubsection]{Remark}

\newcommand{\C}{\mathbb{C}}

\newcommand{\G}{\mathbb{G}}

\newcommand{\R}{\mathbb{R}}

\newcommand{\Z}{\mathbb{Z}}

\newcommand{\afr}{\mathfrak{a}}

\newcommand{\gfr}{\mathfrak{g}}
\newcommand{\hfr}{\mathfrak{h}}

\newcommand{\mfr}{\mathfrak{m}}

\newcommand{\ofr}{\mathfrak{o}}

\newcommand{\ufr}{\mathfrak{u}}

\newcommand{\zfr}{\mathfrak{z}}

\newcommand{\Dcl}{\mathcal{D}}

\newcommand{\Ocl}{\mathcal{O}}

\newcommand{\Scl}{\mathcal{S}}

\newcommand{\floor}[1]{\left\lfloor #1\right\rfloor}
\newcommand{\ceil}[1]{\left\lceil #1 \right\rceil}
\DeclareMathOperator{\Gal}{Gal}

	\newcommand{\stc}{\mathrm{stc}}

	\newcommand{\GL}{\mathrm{GL}}
	\newcommand{\UU}{\mathrm{U}}
	\newcommand{\gl}{\mathfrak{gl}}

\DeclareMathOperator{\Ad}{Ad}

\DeclareMathOperator{\Lie}{Lie}

\newcommand{\Res}[1]{\mathrm{Res}_{#1}}

\newcommand{\df}[1]{\emph{#1}}

\newsavebox{\fmbox}

\DeclareMathOperator{\diag}{diag}

\newcommand{\dd}{\mathrm{d}}

\newcommand{\case}[2]{\noindent {\bfseries Case #1: #2}} 

\begin{document}
%\linenumbers
\title[Exposing Relative Endoscopy]{ {\bfseries Exposing Relative Endoscopy in Unitary Symmetric Spaces } }
\author{Jason K.C. Pol\'ak}
\begin{abstract}
	We introduce a new class of symmetric space orbital integrals important for applications in certain \emph{relative trace formulas} appearing in the theory of automorphic representations. We verify a fundamental lemma for $\UU_2\times\UU_2\hookrightarrow \UU_4$ via an explicit calculation, showing strong evidence that there is a general theory of endoscopy lurking in this situation.
\end{abstract}
\date{\today}
\maketitle
\thispagestyle{empty}
~\\~\\

\tableofcontents
\newpage
\thispagestyle{plain}

\section{Introduction}
\subsection{Endoscopy and Representations}
Endoscopy is the theory that allows one to relate $\kappa$-orbital integrals on a reductive algebraic group to stable orbital integrals on smaller \emph{endoscopic} groups. It has been a key tool in establishing special cases of Langlands functoriality, which relates \emph{automorphic representations} of such a group to its endoscopic friends. Endoscopy owes its existence to the following phenomenon: if $G$ is a connected reductive group over a field $F$ with algebraic closure $\overline{F}$, and we are given an $F$-representation $V$, then two elements $\gamma,\gamma'\in V(F)$ may be $G(\overline{F})$-conjugate but not necessarily $G(F)$-conjugate. The $G(F)$-conjugacy classes in the $G(\overline{F})$-conjugacy class of $\gamma$ are parametrised by the pointed set $\Dcl = \ker[H^1(F,I_\gamma)\to H^1(F,G)]$ where $I_\gamma$ is the centraliser of $\gamma$. In favourable circumstances such as the one in this paper, $\Dcl$ is actually a finite abelian group, and in general one can replace it with an abelian group using the technology in~\cite{borovoiabelian}. If $V$ is the adjoint representation or a representation of a subgroup $H$ of $G$ on an $H$-invariant subspace of $V$, and $\gamma$ is semisimple with torus centraliser and regular with respect to this representation, then we say that $\gamma$ and $\gamma'$ are \emph{stably conjugate} whenever they are $G(\overline{F})$-conjugate. We assume that $\gamma$ is of this type for the rest of the paper since we will not actually need the more general definition of stable conjugacy.

If $F$ is now a complete nonarchimedean local field with ring of integers $\ofr$, $G$ and $V$ are defined over $\ofr$, and $\kappa:\Dcl\to \C^\times$ a character, we can form a \emph{$\kappa$-orbital integral}, which is the sum of orbital integrals
\begin{linenomath*}\begin{align*}
	\Ocl^\kappa_\gamma(\mathbf{1})=\sum_{\gamma'} \kappa(\gamma')\int_{I_{\gamma'}(F)\backslash G(F)} \mathbf{1}(\rho(g)^{-1}\gamma')\frac{\dd g}{\dd t}
\end{align*}\end{linenomath*}
where $\mathbf{1}$ is the characteristic function of $V(\ofr)$ and the Haar measures are chosen appropriately so that the integral points have unit volume. We have abused notation by writing $\kappa(\gamma')$ for $\kappa$ evaluated at the cohomology class corresponding to $\gamma'$. In this special case of $\kappa = 1$, this sum of integrals is called a \emph{stable orbital integral}, and is written $\Scl\Ocl_\gamma$. The fundamental lemma for Lie algebras, established in full generality by Ng\^o~\cite{ngo2010lemme}, states that there exists for each $\kappa$ a reductive algebraic group $H$ over $F$, and for each regular semisimple $\gamma\in \gfr(F)$, there exists a regular semisimple $\gamma_H\in \hfr(F)$ such that
\begin{linenomath*}\begin{align*}
	\Ocl^\kappa_\gamma = \Delta(\gamma,\gamma_H) \Scl\Ocl_{\gamma_H}.
\end{align*}\end{linenomath*}
where the representation is the adjoint representation.

Here, $\Delta(\gamma,\gamma_H)\in\C$ is some factor that depends only on $\gamma$ and $\gamma_H$, and is in fact a power of $q$ up to a root of unity. This identity allows one to compare the stabilised trace formula on $G$ and on $H$, which has led to spectacular applications in establishing special cases of Langlands functoriality. 

\subsection{In This Paper} We initiate the study of $\kappa$-orbital integrals for a pair $(G,\theta)$ where $G$ is a connected reductive algebraic group over a complete nonarchimedean local field and $\theta:G\to G$ is an involution, important for the theory of various relative trace formula. Now, instead of the adjoint representation, we use the adjoint action of  the identity component of the $\theta$-fixed points, $(G^\theta)^\circ$ on $\gfr_1 =\{ x\in \gfr : \theta(x) = -x\}$. Hence the group $G_0 := (G^\theta)^\circ$ is the focus of attention. In this setting we have a $\kappa$-orbital integral
\begin{linenomath*}\begin{align}\label{eqn:kapparel}
	\Ocl^\kappa_\gamma(\mathbf{1})=\sum_{\gamma'} \kappa(\gamma')\int_{I_{\gamma'}(F)\backslash G_0(F)} \mathbf{1}(\Ad(g)^{-1}\gamma')\frac{\dd g}{\dd t}
\end{align}\end{linenomath*}
We prove a fundamental lemma for $(\UU(4),\theta)$ where $\theta$ is an involution such that $\UU(4)^\theta \cong \UU(2)\times\UU(2)$ and when $\gamma$ is of the form $\gamma = \diag(x,y,-y,-x)$ with $x\not=\pm y\in F^\times$. Motivated by the usual fundamental lemma for unitary groups, we define for a nontrivial $\kappa:\Dcl(I_{\gamma})\to\C^\times$ the endoscopic symmetric space to be $(H,\theta_H) = (\UU_2,\sigma)\times( \UU_2,\sigma)$ where $\sigma:\UU_2\to \UU_2$ is such that $\UU_2^\sigma\cong\UU_1\times\UU_1$. We then set $\gamma_H = \diag(x,-x)\times\diag(y,-y)\in\hfr_1$. We assume $\gamma$ to be \emph{nearly singular}, which by definition means that $v(x+y)> v(x-y)$. Under these assumptions, we prove the following fundamental lemma.
\begin{thm}The $\kappa$-orbital as defined in~\eqref{eqn:kapparel} satisfies
	\begin{linenomath*}
		\begin{align*}
			\Ocl^\kappa_\gamma(\mathbf{1}_{\gfr_1(\ofr)}) = \Delta(\gamma,\gamma_H)\Scl\Ocl_{\gamma_H}(\mathbf{1}_{\hfr_1(\ofr)}).
		\end{align*}
		where $\Delta(\gamma,\gamma_H)\in\C$ can be calculated explicitly and is a simple power of $q$ up to a root of unity.
	\end{linenomath*}
\end{thm}
Even though relative orbital integrals have been considered previously in the literature, this is the \emph{first known example} of endoscopy in this setting, and will be helpful in formulating more general relative endoscopic fundamental lemmas. Our primary motivation is the relative trace formula studied in~\cite{getzw2013}, the original goal of those authors being to produce nontrivial cycles on unitary Shimura varieties.

\subsection{Acknowledgements}The author would like to thank Jayce R. Getz for suggesting this problem, a thorough reading of the manuscript, and for encouragement.

\section{Conventions and Convenient Facts}

\subsection{Fields, Groups, and Haar Measures}\label{sec:introfirst}
Let $F$ be a complete nonarchimedean local field of \emph{odd} positive characteristic with algebraic closure $\overline{F}$ and $E/F$ an unramified quadratic extension. We denote the nontrivial action of the Galois group by $x\mapsto\overline{x}$. Since $E/F$ is fixed throughout, we simply use $N(x) = x\overline{x}$ to denote the norm of $x$. We fix once and for all a $\delta\in E$ such that $\overline{\delta} = -\delta$ and $v(\delta)=0$ so that $F\delta = \{ x\in E : \overline{x} = -x\}$. Let $\ofr\subset F$ be the ring of integers with maximal ideal $\mfr$, and let $\ofr_E$ be the ring of integers of $E$. We write $q = |\ofr/\mfr|$, the cardinality of the residue field. We fix once and for all a uniformiser $\pi$, and denote the resulting valuation on $F$ by $v:F^\times\to\Z$ so that $v(\pi) = 1$.

In our computations we will consider various Haar measures on locally compact groups of the form $G(F)$ where $G$ is a linear algebraic group over $\ofr$. These are always normalised so that $G(\ofr)$ has unit volume.  We will need the following proposition which follows from $\ofr_E$ being stable under $\Gal(E/F)$.
\begin{prop}\label{thm:basicintegrality}
	If $ax + b\in\ofr_{E}$ where $a,b\in F$ and $x\in F\delta$ then $ax\in\ofr_E$ and $b\in\ofr_E$.\qedsymbol
\end{prop}
\section{Symmetric Space Representations and Unitary Groups}\label{sec:introunitary}
\subsection{Symmetric Spaces}

In this section we define the symmetric space representation and recall some facts we will need.

\begin{defn}
	Let $\theta:G\to G$ be an involution of a reductive algebraic group over a field $F$. We abuse notation by writing $\theta$ also for the differential of $\theta$. We define the \df{symmetric space representation} to be the representation of $G_0 = (G^\theta)^\circ$ on the $-1$-eigenspace $\gfr_1 = \{ x \in \gfr : \theta(x) = -x\}$.
\end{defn}

The Lie subalgebra $\gfr_0 = \gfr^\theta$ of fixed points also plays and important role. For $x\in \gfr_1$, if $\dim \zfr_{\gfr_0}(x) \leq \dim \zfr_{\gfr_0}(y)$ for all $y\in \gfr_1$, then we say that $x$ is \df{regular}. This is equivalent to the orbit $G_0\cdot x$ having minimal dimension. For this and further facts, the reader is referred to the paper~\cite{levyInvolutions}.

\subsection{$\UU_2\times\UU_2\hookrightarrow \UU_4$ and Regular Elements}
Recall that $F$ is a complete nonarchimedean local field of positive characteristic and $E$ is an unramified quadratic extension (see \S\ref{sec:introfirst}). Define the $n\times n$ matrix
\begin{linenomath*}\begin{align*}J_n=
	\begin{pmatrix}
		~ & ~ & 1\\
		~ & \diagup & ~\\
		1 & ~ & ~
	\end{pmatrix}.
\end{align*}\end{linenomath*}
The $n\times n$ unitary group functor is defined for all $F$-algebras $R$ by
\begin{linenomath*}\begin{align*}
	\UU_n(R) = \{ g\in \GL_4(R\otimes_F E) : J_n\overline{g}^{-t}J_n = g\}
\end{align*}\end{linenomath*}
and its Lie algebra is the functor given by
\begin{linenomath*}\begin{align*}
	\gfr:= \ufr_n(R) = \{ x\in \gl_n(R\otimes_F E): -J_n\overline{x}^t J_n = x\}.
\end{align*}\end{linenomath*}
From now on we consider the case $n =4$. We let $\theta:\Res{E/F}\GL_4\to\Res{E/F}\GL_4$ be conjugation by 
\begin{linenomath*}\begin{align*}
	\theta=\begin{pmatrix}
		& &  &1\\
		& & -1 &\\
		& -1 & & \\
		1 &  & & 
	\end{pmatrix}
\end{align*}\end{linenomath*}
which by abuse of notation we also call $\theta$. Since $\theta J = J\theta$, the involution $\theta$ on $\Res{E/F}\GL_4$ gives a well-defined involution on $\UU_4$. For computational purposes, it is necessary to write down explicitly the $F$-points of the $-1$-eigenspace $\gfr_1$ in terms of matrices. This can be most easily done by observing that
\begin{linenomath*}\begin{align*}
	\gl_4(E)(-1) = \{ x\in \gl_4(E) : \theta(x) = -x\} = \left\{ 
	\begin{pmatrix}
		x_{11} & x_{12} & x_{13} & x_{14} \\
		x_{21} & x_{22} & x_{23} & x_{24} \\
		x_{24} & -x_{23} & -x_{22} & x_{21} \\
		-x_{14} & x_{13} & x_{12} & -x_{11}
	\end{pmatrix} : x_{ij}\in E
	\right\}
\end{align*}\end{linenomath*}
so that $\gfr_1$ is then the fixed points under $x\mapsto -J_4\overline{x}^t J_4$, and the $F$-points are then
\begin{linenomath*}\begin{align}\label{eqn:liem1}
	\gfr_1(F) = \left\{ 
	\begin{pmatrix}
		x_{11} & x_{12} & x_{13} & x_{14} \\
		-\overline{x}_{12} & x_{22} & x_{23} & -\overline{x}_{13} \\
		-\overline{x}_{13} & -x_{23} & -x_{22} & -\overline{x}_{12} \\
		-x_{14} & x_{13} & x_{12} & -x_{11}
	\end{pmatrix} : \begin{matrix}x_{11},x_{22}\in F\\ x_{14},x_{23}\in F\delta\end{matrix}
	\right\}.
\end{align}\end{linenomath*}
The fixed-point group $\UU_4^\theta \cong \UU_2\times\UU_2 =: G_0$ acts on $\gfr_1$. Inside $\gfr_1(F)$ there is $\afr$, a maximal subspace of commuting semisimple elements such that
\begin{linenomath*}\begin{align*}
	\afr(F) = \left\{
	\begin{pmatrix}
		x & 0 & 0 & 0\\
		0 & y & 0 & 0\\
		0 & 0 & -y & 0\\
		0 & 0 & 0 & -x
	\end{pmatrix} : x,y\in F\right\}
\end{align*}\end{linenomath*}
We only consider regular semisimple elements in $\afr(F)$. If $\gamma = \diag(x,y,-y,-x)\in \afr(F)$, then $\gamma$ is regular if and only if $x,y\in F^\times$ and $x\not=\pm y$. For any regular semisimple element $\gamma\in\gfr_1(F)$ with centraliser $I_\gamma$ in $\UU_2\times\UU_2$, and for any compactly supported complex-valued smooth function $f$ on $\gfr_1(F)$, we define the orbital integral
\begin{linenomath*}\begin{align*}
	\Ocl_\gamma(f) := \int_{I_\gamma(F)\backslash \UU_2\times\UU_2(F)} f(\Ad(g^{-1})\gamma)\frac{\dd g}{\dd g_\gamma}.
\end{align*}\end{linenomath*}

\subsection{Stable Conjugacy} Fix a regular $\gamma = \diag(x,y,-y,-x)\in\afr(F)$. The stable conjugacy class of $\gamma$ in $\gfr_1(F)$ by definition is $G_0(\overline{F})\gamma\cap \gfr_1(F)$. When using cohomology, it is useful to express the stable conjugacy class of $\gamma$ as
\begin{linenomath*}\begin{align*}
	\{ \Ad(g)\gamma : g\in G_0(\overline{F})\text{~and~}g^{-1}\sigma(g)\in I_\gamma(\overline{F})\text{~for all~} \sigma\in \Gal(\overline{F}/F) \}.
\end{align*}\end{linenomath*}
In this section, we explicitly decompose the stable conjugacy class of $\gamma$ into $G_0(F)$-conjugacy classes. As before, we denote by $I_\gamma$ the stabiliser of $\gamma$ in $G_0$. The inclusion $I_\gamma\to G_0$ gives rise to a long exact sequence of pointed sets
\begin{linenomath*}\begin{align*}
	1\to I_\gamma(F)\to G_0(F)\to (G_0/I_\gamma)(F)\to H^1(F,I_\gamma)\to H^1(F,G_0)
\end{align*}\end{linenomath*}
in nonabelian cohomology. One easily checks that the map $\Ad(g)\gamma\mapsto (\sigma\mapsto g^{-1}\sigma(g))$ is a well-defined bijection from the set of conjugacy classes of $\gamma$ in the stable conjugacy class to $\Dcl:=\ker[H^1(F,I_\gamma)\to H^1(F,G_0)]$. Using this correspondence, we can compute explicit representatives for the conjugacy classes within the stable conjugacy class of $\gamma$.

We first compute
\begin{linenomath*}\begin{align*}
	I_\gamma = \left\{ 
	\begin{pmatrix}
		a_{11} & 0 & 0 & 0\\
		0 & a_{22} & 0 & 0\\
		0 & 0 & a_{22} & 0\\
		0 & 0 & 0 & a_{11}
	\end{pmatrix} : a_{ii}\overline{a}_{ii} = 1
	\right\}.
\end{align*}\end{linenomath*}
In other words, $I_\gamma \cong \UU_1\times\UU_1$ (which suggests our choice of an endoscopic symmetric space). We can compute the Galois cohomology over the finite extension $E/F$ which splits $I_\gamma$. Doing this we find that
\begin{linenomath*}\begin{align*}
	H^1(F,I_\gamma) \cong \Z/2\times\Z/2.
\end{align*}\end{linenomath*}
A quick application of Tate-Nakayama duality shows that the kernel $\Dcl$ of $H^1(F,I_\gamma)\to H^1(F,G_0)$ is $\Z/2$. For the purposes of computations, let us be more explicit. Fix the isomorphism of algebras $E\otimes_F E\xrightarrow{\sim} E\oplus E$ given on pure tensors by $(a\otimes b)\mapsto (ab,\overline{a}b)$, where the multiplication on $E\oplus E$ is pointwise and the Galois action on the left factor of $E\otimes_FE$ translates to $\overline{(a,b)} = (b,a)$ in $E\oplus E$. 

Let $\sigma\in\Gal(E/F)$ be the nontrivial element. The nontrivial element in $\Dcl$ is then represented (for example) by the cocycle
\begin{linenomath*}
\begin{align*}
\sigma\mapsto  (\pi I_4,\pi^{-1}I_4)\in I_\gamma(E)
\end{align*}
\end{linenomath*}
The element in $(I_\gamma\backslash G_0)(E)$ that maps to the corresponding cohomology class is represented by $(B,\pi B)\in G_0(E)$ where
\begin{linenomath*}\begin{align*}
	B = \frac 12
	\begin{pmatrix}
\pi^{-1} + 1 & -\pi^{-1} + 1 & 0 & 0 \\
-\pi^{-1} + 1 & \pi^{-1} + 1 & 0 & 0 \\
0 & 0 & \pi^{-1} + 1 & \pi^{-1} - 1 \\
0 & 0 & \pi^{-1} - 1 & \pi^{-1} + 1	\end{pmatrix}.
\end{align*}\end{linenomath*}
Indeed, it maps to $\sigma\mapsto g^{-1}\sigma(g) = (\pi I_4,\pi^{-1}I_4)$. We remark to the unwary reader that $\sigma$ here acts on the \emph{right} of $E\otimes_F E$, which on $E\oplus E$ is the same $(a,b)\mapsto (\overline{b},\overline{a})$. Now $H^1(F,I_\gamma)\cong H^1(\Z/2,I_\gamma(E))$ and $I_\gamma(E) \cong \{ \diag( (a,a^{-1}),(b,b^{-1}),(b,b^{-1}),(a,a^{-1}) ) : a,b\in E^\times\}$, so we are just computing group cohomology of a cyclic group. We get
\begin{linenomath*}
\begin{align*}
	H^1(\Z/2,I_\gamma(E)) &\cong (E^\times/(E^\times)^2)\times (E^\times/(E^\times)^2)\\
	&\cong \Z/2\times \Z/2.
\end{align*}
\end{linenomath*}
Hence $\sigma\mapsto (\pi I_4,\pi^{-1}I_4)$ is a nontrivial cocycle because $\pi$ is not a square in $E$. We set $\gamma_\stc = \Ad(B)\gamma$. Now let $\kappa:\Dcl\to \{\pm1\}$ be a character. We have for any compactly supported smooth function $f$ on $\gfr_1(F)$ the $\kappa$-orbital integral
\begin{linenomath*}\begin{align*}
	\Ocl_\gamma^\kappa(f) = \int_{I_\gamma(F)\backslash G_0(F)} f(\Ad(g^{-1})\gamma)\dd g + \kappa(-1)\int_{I_{\gamma_\stc}(F)\backslash G_0(F)} f(\Ad(g^{-1})\gamma_\stc)\dd g.
\end{align*}\end{linenomath*}
We note that we can omit the centralisers of $\gamma$ and of $\gamma_\stc$ since their $F$-points are compact. We would like to compute this integral when $\kappa$ is the nontrivial character and when $f=\mathbf{1}$, the characteristic function of $\gfr_1(\ofr)$:
\begin{linenomath*}\begin{align*}
	\Ocl_\gamma^\kappa(\mathbf{1}) = \int_{G_0(F)} \mathbf{1}(\Ad(g^{-1})\gamma)\dd g - \int_{G_0(F)} \mathbf{1}(\Ad(g^{-1})\gamma_\stc)\dd g.
\end{align*}\end{linenomath*}
This is the goal of the remainder of the paper.
\subsection{Preliminaries on Integration}\label{sec:integPrelims}  We choose a parabolic $P$ so that we get an Iwasawa decomposition $G_0(F) = P(F)G_0(\ofr) = M(F)U(F)G_0(\ofr)$, where $M$ is a Levi subgroup of $P$ and $U$ is the unipotent radical of $P$, so we reduce the computation to one on $P(F) = M(F)U(F)$. Although this does simplify matters, since the centraliser of $\gamma$ is compact, we unfortunately \emph{cannot} use the method of descent that would otherwise make the computation significantly easier. Now, to specify a parabolic of $G_0$ is the same thing as giving a cocharacter $\lambda:\G_m\to G_0$. We use the cocharacter
\begin{linenomath*}\begin{align*}
	\G_m &\longrightarrow G_0\\
	r&\longmapsto \frac 12
	\begin{pmatrix}
		r + r^{-1} & 0 & -r + r^{-1} & 0 \\
0 & r + r^{-1} & 0 & r - r^{-1} \\
-r + r^{-1} & 0 & r + r^{-1} & 0 \\
0 & r - r^{-1} & 0 & r + r^{-1}
	\end{pmatrix}.
\end{align*}\end{linenomath*}
One verifies easily that $\lambda$ is both $\theta$-fixed and actually does land in $\UU_4$. This cocharacter uniquely specifies the parabolic $P = \{ g\in G_0 : \lim_{r\to 0} \lambda(r)g\lambda(r^{-1}) \text{~exists~}\}$. The unipotent radical of this parabolic is given by
$U = \{ g\in G_0 : \lim_{r\to 0} \lambda(r)g\lambda(r^{-1}) =1\}$
and we calculate the $F$-points of the unipotent radical to be isomorphic to $F\delta\times F\delta$ via
\begin{equation}\label{eqn:formUrad}\begin{split}
	F\delta\times F\delta &\longrightarrow U(F)\\
	(c,d)&\longmapsto
	\begin{pmatrix}
		c + 1 & d & c & -d \\
-d & -c + 1 & -d & c \\
-c & -d & -c + 1 & d \\
-d & -c & -d & c + 1
	\end{pmatrix}.\end{split}
\end{equation}
For integrating, we use the product Haar measure on $F\delta\times F\delta$ where on each factor $F\delta$ we choose a Haar measure so that $\ofr\delta$ has unit volume. A Levi subgroup of $P$ is the subgroup that centralises the cocharacter $\lambda$. We calculate its $F$-points to be isomorphic to $E^\times\times E^\times$ through the isomorphism
\begin{equation}\label{eqn:formLevi}\begin{split}
	E^\times \times E^\times&\longrightarrow M(F)\\
	(r_1,r_2)&\longmapsto\frac 14
	\left( \begin{smallmatrix}
	r_{1} + r_{2} + \overline{r}_1^{-1} + \overline{r}_2^{-1} & r_{1} + r_{2} - \overline{r}_1^{-1} - \overline{r}_2^{-1} & r_{1} - r_{2} - \overline{r}_1^{-1} + \overline{r}_2^{-1} & -r_{1} + r_{2} - \overline{r}_1^{-1} + \overline{r}_2^{-1} \\
r_{1} + r_{2} - \overline{r}_1^{-1} - \overline{r}_2^{-1} & r_{1} + r_{2} + \overline{r}_1^{-1} + \overline{r}_2^{-1} & r_{1} - r_{2} + \overline{r}_1^{-1} - \overline{r}_2^{-1} & -r_{1} + r_{2} + \overline{r}_1^{-1} - \overline{r}_2^{-1} \\
r_{1} - r_{2} - \overline{r}_1^{-1} + \overline{r}_2^{-1} & r_{1} - r_{2} + \overline{r}_1^{-1} - \overline{r}_2^{-1} & r_{1} + r_{2} + \overline{r}_1^{-1} + \overline{r}_2^{-1} & -r_{1} - r_{2} + \overline{r}_1^{-1} + \overline{r}_2^{-1} \\
-r_{1} + r_{2} - \overline{r}_1^{-1} + \overline{r}_2^{-1} & -r_{1} + r_{2} + \overline{r}_1^{-1} - \overline{r}_2^{-1} & -r_{1} - r_{2} + \overline{r}_1^{-1} + \overline{r}_2^{-1} & r_{1} + r_{2} + \overline{r}_1^{-1} + \overline{r}_2^{-1}	
	\end{smallmatrix}\right) 
\end{split}\end{equation}
Again, for integration, we use the product Haar measure on $E^\times\times E^\times$ so that $\ofr_E^\times$ in $E^\times$ has unit volume. 

We note that multiplying the matrix in~\ref{eqn:formLevi} either on the left or the right by the matrix that represents the cocycle is the same matrix but with $\overline{r}_i$ replaced with $\pi\overline{r}_i$ for $i=1,2$. Thus, in any expressions that are a function of $m\gamma m^{-1}$ for $m\in M$, making this replacement gives us the equations for $m\gamma_\stc m^{-1}$.

In making these reductions, we are left to evaluate the integral
\begin{linenomath*}\begin{align*}
	\Ocl_\gamma^\kappa(\mathbf{1}) = \int_{M(F)}\int_{U(F)} \mathbf{1}(\Ad(u^{-1})\Ad(m^{-1})\gamma)\dd u\dd m - \int_{M(F)}\int_{U(F)} \mathbf{1}(\Ad(u^{-1})\Ad(m^{-1})\gamma_\stc)\dd u\dd m
\end{align*}\end{linenomath*}
where the Haar measures are chosen so that $U(\ofr)$ and $M(\ofr)$ each have unit volume. We now examine an element of the form $Y = \Ad(u^{-1})\Ad(m^{-1})\gamma$, where $u$ is a matrix as in \eqref{eqn:formUrad} and $m$ is a matrix as in \eqref{eqn:formLevi}. Since $Y\in\gfr_1$, by the explicit form of $\gfr_1$ in \eqref{eqn:liem1}, we see that $Y\in\gfr_1(\ofr)$ exactly when $v = [Y_{11},Y_{12},Y_{13},Y_{14},Y_{22},Y_{23}]^t\in\ofr^6$. And, $v\in \ofr^6$ exactly when $Av\in\ofr^6$ for any $A\in \GL_6(\ofr)$. In particular, we take
\begin{linenomath*}\begin{align*}
			A = 
		\begin{pmatrix}
			0 & -1 & 1 & 0 & -1 & 1 \\
1 & -1 & -1 & -1 & 0 & 0 \\
0 & 1 & 1 & 0 & -1 & -1 \\
1 & 1 & -1 & 1 & 0 & 0 \\
1 & 1 & 1 & -1 & 0 & 0 \\
1 & -1 & 1 & 1 & 0 & 0
		\end{pmatrix}, & & Av =  	\begin{pmatrix}
			-Y_{12} + Y_{13} - Y_{22} + Y_{23}\\
			Y_{11} - Y_{12} - Y_{13} - Y_{14}\\
			Y_{12} + Y_{13} - Y_{22} - Y_{23}\\
			Y_{11} + Y_{12} - Y_{13} + Y_{14}\\
			Y_{11} + Y_{12} + Y_{13} - Y_{14}\\
			Y_{11} + Y_{13} - Y_{12} + Y_{14}
		\end{pmatrix}.
\end{align*}\end{linenomath*}
We calculate that $\det(A) = -32$, so that indeed $A\in\GL_6(\ofr)$. Now it is simply a matter of calculating $Y= \Ad(u^{-1})\Ad(m^{-1})\gamma$, and each of the quantities in $Av$. Straightforward but tedious computations, and making the harmless change of variables ($d-c\rightsquigarrow c, d + c\rightsquigarrow d$) show that
\begin{linenomath*}\begin{align}
	\label{eqn:o1}-Y_{12} + Y_{13} - Y_{22} + Y_{23} &= \tfrac{r_2}{r_1}(x+y)(-c-\tfrac 12) + \tfrac 12(x-y)\tfrac{1}{r_1\overline{r_2}}\\
	\label{eqn:o2}\overline{Y}_{11} - \overline{Y}_{12} - \overline{Y}_{13} - \overline{Y}_{14} &= \tfrac{r_2}{r_1}(x+y)(-c+\tfrac 12) + \tfrac 12(x-y)\tfrac{1}{r_1\overline{r}_2}\\
	\label{eqn:o3}\overline{Y}_{12} + \overline{Y}_{13} - \overline{Y}_{22} - \overline{Y}_{23} &= \tfrac{r_2}{r_1}(x+y)(-d-\tfrac 12) + \tfrac 12(x-y)\overline{r}_1r_2\\
	\label{eqn:o4}Y_{11} + Y_{12} - Y_{13} + Y_{14} &= \tfrac{r_2}{r_1}(x+y)(-d+\tfrac 12) + \tfrac 12(x-y)\overline{r}_1r_2\\
	\begin{split}
		\label{eqn:o5}\overline{Y}_{11} + \overline{Y}_{12} + \overline{Y}_{13} - \overline{Y}_{14} &= 2 \tfrac{r_2}{r_1}cd(x+y) - d\left[ \tfrac{r_2}{r_1}(x+y) + \tfrac{1}{r_1\overline{r}_2}(x-y) \right] \\&- \overline{r}_1r_2c(x-y) + \tfrac 12\left[ \overline{r}_1r_2(x-y) + \tfrac{\overline{r}_1}{\overline{r}_2}(x+y) \right]
\end{split}\\
\begin{split}\label{eqn:o6}
	Y_{11} + Y_{13} - Y_{12} + Y_{14} &=  2\tfrac{r_2}{r_1}cd(x+y)-c\left[ \tfrac{r_2}{r_1}(x+y) + \overline{r}_1r_2(x-y) \right] \\&- \tfrac{1}{r_1\overline{r}_2}d(x-y) + \tfrac 12\left[ \tfrac{\overline{r}_1}{\overline{r}_2}(x+y) + \tfrac{1}{r_1\overline{r}_2}(x-y) \right]
\end{split}
\end{align}\end{linenomath*}
We note that we have also harmlessly replaced some terms by their conjugates. We now simplify these terms even further, preserving their status of integrality. Subtracting \eqref{eqn:o1} from \eqref{eqn:o2} shows that the integrality of these implies  that $\tfrac{r_2}{r_1}(x+y)$ is integral. In particular, it thus follows that if \eqref{eqn:o1} is integral then so is
\begin{linenomath*}\begin{align}\label{eqn:intermed1}
	\tfrac{r_2}{r_1}(x+y)(-c) + \tfrac 12(x-y)\tfrac{1}{r_1\overline{r}_2}.
\end{align}\end{linenomath*}
Multiplying this expression by the valuation-zero term $\pi^{-(v(r_1)+v(r_2))}r_1\overline{r}_2$ and applying Proposition~\ref{thm:basicintegrality} shows that each term of this expression is in fact integral exactly when the entire expression is integral. We can of course apply the same reasoning to \eqref{eqn:o3} and \eqref{eqn:o4}, which shows that \eqref{eqn:o1}-\eqref{eqn:o4} being integral is equivalent to the following expressions being integral:
\begin{linenomath*}\begin{align*}
	E_1 &= \tfrac{r_2}{r_1}(x+y) & 	E_4 &= (x-y)\tfrac{1}{r_1\overline{r}_2}\\
	E_2 &= \tfrac{r_2}{r_1}(x+y)c &		E_5 &= (x-y)\overline{r}_1r_2\\
	E_3 &= \tfrac{r_2}{r_1}(x+y)d\\
\end{align*}\end{linenomath*}
We observe also that if we subtract \eqref{eqn:o4} from \eqref{eqn:o5}, we get the same thing as subtracting \eqref{eqn:o2} from \eqref{eqn:o6}, both differences being equal to
\begin{linenomath*}\begin{align}\label{eqn:o5o6replacement}
	2\tfrac{r_2}{r_1}cd(x+y) -(x-y)\left[ c\overline{r}_1r_2 + \tfrac{d}{r_1\overline{r}_2} \right] + \tfrac{1}{2}\tfrac{\overline{r}_1}{\overline{r}_2}(x+y)
\end{align}\end{linenomath*}
so we might as well replace \eqref{eqn:o5} and \eqref{eqn:o6} by \eqref{eqn:o5o6replacement}. We can again apply Proposition~\ref{thm:basicintegrality} to \eqref{eqn:o5o6replacement} multiplied by $\pi^{-(v(r_1)+v(r_2))}r_1\overline{r}_2$, which tells us that the integrality of \eqref{eqn:o5o6replacement} is actually equivalent to the integrality of these two:
\begin{linenomath*}\begin{align*}
	E_6 &= 4\tfrac{r_2}{r_1}cd(x+y) + \tfrac{\overline{r}_1}{\overline{r_2}}(x+y)\\
	E_7 &= (x-y)\left[ c\overline{r}_1r_2 + \tfrac{d}{r_1\overline{r}_2} \right]
\end{align*}\end{linenomath*}
We have come to the end of our simplifications on the conditions that determine whether $Y = \Ad(u^{-1})\Ad(m^{-1})\gamma$ is integral.
\newpage
\subsection{Elimination of Fiendish Cases}\label{sec:elimination}

We have determined in \S\ref{sec:integPrelims} expressions $E_1,E_2,\dots,E_7$ that are integral exactly when $\Ad(u^{-1})\Ad(m^{-1})\gamma$ is integral. We denote by $E_1^\stc,E_2^\stc,\dots,E_7^\stc$ the corresponding equations for $\gamma_\stc$. Recall that to get the conditions for $\gamma_\stc$, we just replace $\overline{r}_i$ by $\pi\overline{r}_i$ for $i=1,2$ in $E_1,\dots,E_7$. For the remainder of the paper, we set
\begin{linenomath*}\begin{align*}
h &= v(r_1) + v(r_2),\\
V_m &= v(x-y)\\ V_p &= v(x+y).
\end{align*}\end{linenomath*}
To aid the reader, we list $E_1,\dots,E_7$ and their stably conjugate versions, along with their valuations:
\begin{linenomath*}\begin{align*}
	E_1 &= \tfrac{r_2}{r_1}(x+y) & E_1^\stc &= \tfrac{r_2}{r_1}(x+y)\\
	v(E_1) &= h + V_p - 2v(r_1) & v(E_1^\stc) &= h + V_p - 2v(r_1)\\
	E_2 &= \tfrac{r_2}{r_1}(x+y)c& E_2^\stc &= \tfrac{r_2}{r_1}(x+y)c\\
	v(E_2) &= h + V_p + v(c) - 2v(r_1) & v(E_2^\stc) &= h + V_p +v(c)- 2v(r_1)\\
	E_3 &= \tfrac{r_2}{r_1}(x+y)d& E_3^\stc &= \tfrac{r_2}{r_1}(x+y)d\\
	v(E_3) &= h + V_p  + v(d)- 2v(r_1) & v(E_3^\stc) &= h + V_p + v(d)- 2v(r_1)\\
	E_4 &= (x-y)\tfrac{1}{r_1\overline{r}_2} &E_4^\stc &= (x-y)\tfrac{1}{\pi r_1\overline{r}_2}\\
	v(E_4) &= V_m - h & v(E_4^\stc) &= V_m - h - 1\\
	E_5 &= (x-y)\overline{r}_1r_2 & E_5^\stc &= (x-y)\pi\overline{r}_1r_2\\
	v(E_5) &= V_m + h & v(E_5) &= V_m + h + 1\\
	E_6 &= 4\tfrac{r_2}{r_1}cd(x+y) + \tfrac{\overline{r}_1}{\overline{r_2}}(x+y) & 
	E_6^\stc &= 4\tfrac{r_2}{r_1}cd(x+y) + \tfrac{\overline{r}_1}{\overline{r_2}}(x+y)\\
	v(E_6)&\geq \min\left\{ \begin{smallmatrix}h+V_p + v(c) + v(d) - 2v(r_1),\\ V_p - h + 2v(r_1) \end{smallmatrix}\right\} & v(E_6^\stc) &\geq \min\left\{\begin{smallmatrix} h+V_p + v(c) + v(d) - 2v(r_1),\\ V_p - h + 2v(r_1)\end{smallmatrix}\right\}\\
	E_7 &= (x-y)\left[ c\overline{r}_1r_2 + \tfrac{d}{r_1\overline{r}_2} \right] & E_7^\stc &= (x-y)\left[ c\pi\overline{r}_1r_2 + \tfrac{d}{\pi r_1\overline{r}_2} \right]\\
	v(E_7)&\geq \min\left\{ \begin{smallmatrix}V_m + v(c) + h,\\ V_m + v(d) - h\end{smallmatrix} \right\} & v(E_7^\stc) &\geq \min\left\{ \begin{smallmatrix}V_m + v(c) + h + 1,\\ V_m + v(d) - h - 1\end{smallmatrix} \right\}
\end{align*}\end{linenomath*}
We see that only the fourth, fifth, and seventh expressions differ between $\gamma$ and $\gamma_\stc$. We observe that the difficulties will occur mainly with the sixth and seventh, since they are sums. \emph{In this section}, we eliminate some of the more fiendish difficulties to prepare the way for the main calculation in \S\ref{sec:calculations}. In order to lessen the wordiness and symbolism in the sequel, the \emph{summands of $E_6$} will refer to the two terms $4\tfrac{r_2}{r_1}cd(x+y)$ and $\tfrac{\overline{r}_1}{\overline{r_2}}(x+y)$. Similarly, the \emph{summands of $E_7$} will refer to the two terms: $(x-y)c\overline{r}_1r_2$ and $(x-y)\tfrac{d}{r_1\overline{r}_2}$. We also use this terminology, suitably modified, for the stably conjugate versions. For example, if $E_6$ is integral, then we know that either both summands are integral or neither are. These possibilities for $E_6$ and $E_7$ break down the computation into various cases, and the next proposition shows that the worst of these actually cannot occur.
\begin{lem}\label{thm:worstcase}
	If none of the summands in $E_6$ and $E_7$ are integral, then $E_6$ and $E_7$ cannot be simultaneously integral.
\end{lem}
\begin{proof}
	We proceed by contradiction, assuming that none of the summands in $E_6$ and $E_7$ are integral, but that both $E_6$ and $E_7$ are integral. In this case, the valuation of the first summand must be equal to the valuation of the second in $E_6$, and the same is true of $E_7$. Thus, we get a pair of equations:
	\begin{linenomath*}\begin{align*}
		2[v(r_2) - v(r_1)] + v(c) + v(d) &= 0,\\
		2[v(r_2) + v(r_1)] + v(c) - v(d) &= 0.
	\end{align*}\end{linenomath*}
	In particular, $v(c) = -2v(r_2)$ and $v(d) = 2v(r_1)$. Hence
	\begin{linenomath*}\begin{align*}
		v\left( \tfrac{r_2}{r_1}d(x+y)\pi^{-(h+V_p+V_m)} \right) = -V_m.
	\end{align*}\end{linenomath*}
	Multiply $E_6$ by the inverse of the expression in $v(-)$ to get
	\begin{linenomath*}\begin{align}\label{eqn:proof1eq1}
		4c\pi^{h+V_p+V_m} + \frac{N(r_1)\pi^{h+V_p+V_m}}{dN(r_2)}\in \pi^{V_m}\ofr.
	\end{align}\end{linenomath*}
	Using the same procedure on $E_7$ gives
	\begin{linenomath*}\begin{align}\label{eqn:proof1eq2}
		4c\pi^{h+V_m+V_p} + \frac{4d\pi^{h+V_m+V_p}}{N(r_1)N(r_2)}\in\pi^{V_p}\ofr.
	\end{align}\end{linenomath*}
	We take the difference between \eqref{eqn:proof1eq1} and \eqref{eqn:proof1eq2}, obtaining
	\begin{linenomath*}\begin{align}\label{eqn:proof1eq3}
		\frac{\pi^{h+V_m+V_p}}{dN(r_1)N(r_2)}(N(r_1) - 2d)(N(r_1)+ 2d).
	\end{align}\end{linenomath*}
	We note that $N(r_1)\in F$ whereas $2d\in F\delta$. Hence, the valuation of $N(r_1)\pm 2d$ is precisely $v(d)$. Thus, the valuation of \eqref{eqn:proof1eq3} is $V_m + V_p + v(d) - h$. There are two cases to consider: either $V_p \geq V_m$ or $V_p < V_m$.

	\noindent {\bfseries Case 1: $V_p \geq V_m$.} Then \eqref{eqn:proof1eq3} lies in $\pi^{V_m}\ofr$, or in other words, $v(d) + V_p + V_m - h \geq V_m$. Simplifying, we get $v(d) + V_p - h\geq 0$. On the other hand, the first summand of $E_6$ also has valuation $v(d) + V_p - h$, showing that this summand is integral, which is a contradiction.

	\noindent {\bfseries Case 2: $V_p < V_m$}. Then \eqref{eqn:proof1eq3} lies in $\pi^{V_p}\ofr$, or in other words, $v(d) + V_p + V_m - h \geq V_p$. Hence $v(d) + V_m - h\geq 0$, but this is the valuation of the second summand of $E_7$, which is again a contradiction.
\end{proof}
\begin{rem}
	In this paper, Case 2 in the above proof does not actually occur since we will assume for the actual computation that $V_p > V_m$, but we have included the more general statement for completeness.
\end{rem}
The reader will have no trouble applying the same argument to prove the stably conjugate version.
\begin{lem}\label{thm:worstcasestc}
	If none of the summands in $E_6^\stc$ and $E_7^\stc$ are integral, then $E_6^\stc$ and $E_7^\stc$ cannot be simultaneously integral.
\end{lem}
The next lemma allows us a significant simplification if we stick with a ``limiting case'' for $\gamma$.
\begin{lem}\label{thm:secondworstcase}
	Suppose that $V_p > V_m$, that $E_1,E_2,\dots,E_7$ are integral, and that not all the $E_i^\stc$ are integral (the last condition being equivalent to: $E_7^\stc$ is not integral). Under these conditions, if the summands of $E_7$ are integral, then the summands of $E_6$ are integral as well.\qedsymbol
\end{lem}
\begin{proof}
	We suppose by contradiction that we have a solution that makes $E_1$ to $E_7$ integral, that the summands of $E_7$ are integral, but that the summands of $E_6$ are not integral. Before we list the inequalities in this case, let us make three observations.
	\begin{enumerate}
		\item From $E_2$, we get $h + V_p + v(c) - 2v(r_1) \geq 0$. However, the first summand of $E_6$ not being integral is equivalent to $h + V_p  +v(c) + v(d) -2v(r_1) < 0$. Hence $v(d)  < 0$. Repeating the argument with $E_3$ in place of $E_2$ shows $v(c) < 0$.
		\item The valuation of the first summand of $E_6$ is equal to the valuation of the second. This implies that $v(c)  +v(d) = 4v(r_1) - 2h$.
		\item   Since $E_4=E_4^\stc$, any solution of $E_1,\dots,E_7$ will also be a solution of the stably conjugate versions unless $v(d) = h-V_m$, so we evaluate only under this additional condition, and this implies based on our second observation that $v(c) = 4v(r_1) - 3h + V_m$.
	\end{enumerate}
	Hence we have the following inequalities, by using the substitutions in (1)-(3) in $E_2,E_3$, either term of $E_6$, and the first term of $E_7$ being integral:
	\begin{linenomath*}\begin{align*}
		2v(r_1) &\geq 2h - V_p - V_m\\
		2h + V_p - V_m &\geq 2v(r_1)\\
		h - V_p &> 2v(r_1)\\
		2v(r_1) &\geq h  -V_m
	\end{align*}\end{linenomath*}
	Or, more succinctly,
	\begin{linenomath*}\begin{align*}
		\min\{2h + V_p - V_m,h-V_p-1\}\geq 2v(r_1) \geq \max\{2h - V_p - V_m,h-V_m\}.
	\end{align*}\end{linenomath*}
	We see that $2h + V_p - V_m \geq h - V_p - 1$ is equivalent to $h + 2V_p - V_m +1\geq 0$. Since $V_p > V_m$, so that $h + 2V_p - V_m +1 \geq h + V_m + 1\geq 1$. Similarly, $h - V_m\geq 2h - V_m - V_m$ is equivalent to $V_p \geq h$, which is true again since $V_p > V_m$. Hence we have that,
	\begin{linenomath*}\begin{align*}
		h - V_p - 1\geq 2v(r_1) \geq h - V_m.
	\end{align*}\end{linenomath*}
	So $h - V_p - 1\geq h - V_m$, or equivalently, $V_m\geq V_p+1$, which is absurd.
\end{proof}
Again, the same argument will apply to the stably conjugate version.
\begin{lem}\label{thm:secondworstcasestc}
	Suppose that $V_p > V_m$, and that $E_1^\stc,E_2^\stc,\dots,E_7^\stc$ are integral, but that at least one of $E_1,\dots, E_7$ is not integral. If the summands of $E_7^\stc$ are integral, then the summands of $E_6^\stc$ are integral as well.\qedsymbol
\end{lem}

\section{Brute Force Calculations}\label{sec:calculations}
\subsection{Preliminaries}

We introduced in \S\ref{sec:introunitary} the orbital integral
\begin{linenomath*}
	\begin{align}\label{eqn:bruteorbits}
	\Ocl_\gamma^\kappa = \int_{G_0(F)}\mathbf{1}(\Ad(g)^{-1}\gamma)\dd g - \int_{G_0(F)}\mathbf{1}(\Ad(g)^{-1}\gamma_\stc)\dd g.
\end{align}
\end{linenomath*}
The integral $\int_{G_0(F)}\mathbf{1}(\Ad(g^{-1}\gamma)\dd g$ is the same as the measure of the set
\begin{align*}
	\{ (c,d,r_1,r_2) \in F\delta\times F\delta\times E^\times\times E^\times : E_1,\dots,E_7\text{~are integral~}\},
\end{align*}
and the analogous statement holds for the stably conjugate version. To evaluate the orbital integral, we do as follows: first, we find the measure of the subset of $F\delta\times F\delta\times E^\times\times E^\times$ such that all the $E_*$ are integral but at least one of $E_*^\stc$ is not integral. We refer to this as the $(1,0)$-case. Similarly, the $(0,1)$-case is when all of $E_*^\stc$ are integral but at least one of $E_*$ is not. We then take the measure of the $(1,0)$ case and subtract the measure of the $(0,1)$ case. 

In this section we carry out the calculation, evaluating the integral. Our strategy is to fix $h = v(r_1) + v(r_2)$, write down an expression for the measure of the solution set, and then sum over all the possibilities for $h$:

\begin{lem}
	The integrality of $E_4$ and $E_5$ is equivalent to the inequality $V_m \geq h \geq -V_m$. Similarly, the integrality of $E_4^\stc$ and $E_5^\stc$ is equivalent to $V_m-1\geq h \geq -V_m-1$. In either case, if $V_m < 0$, then neither of this equalities can be satisfied and hence the $\kappa$-orbital integral vanishes.\qedsymbol
\end{lem}

\begin{defn}
	We say that $\gamma$ is \emph{nearly singular} if $V_p > V_m$. 
\end{defn}

For the rest of this paper, we assume that $\gamma$ is nearly singular, which is relatively harmless since our calculation under this assumption still should give us the correct transfer factor, assuming that there is a sane version of endsocopy operating in the midst. At any rate, in view of Lemmas \ref{thm:worstcase},\ref{thm:worstcasestc},\ref{thm:secondworstcase}, and \ref{thm:secondworstcasestc}, we then have to consider two possibilities: all summands in $E_6$ and in $E_7$ are integral, and the summands of $E_6$ are integral but the summands of $E_7$ are not.

\subsection{Integer Arithmetic}\label{sec:intArithmetic} Here we state the properties of floor and ceiling functions we use. For any $r\in\R$ we write $\floor{r}$ and $\ceil{r}$ for the floor of $r$ and the ceiling of $r$ respectively. If $a,b\in \Z$, then we will frequently need the following facts that are easy to verify, but included for convenience in following lengthy computations:
\begin{linenomath*}\begin{align*}
	\bigl|\{ x\in \Z : a \geq 2x \geq b\}\bigr| &= \floor{\frac a 2} - \ceil{\frac b 2} + 1 =
	\begin{cases}
		\tfrac{a - b}{2} + 1 & \text{if $a,b$ are even}\\
		\tfrac{a-b}{2} & \text{if $a,b$ are odd}\\
		\tfrac{a-b+1}{2} & \text{if $a,b$ have opposite parity}
	\end{cases}\\
	\floor{\frac a 2} &= \ceil{\frac{a-1}2}\\
	\floor{\frac a 2}+1 &= \floor{\frac {a+2}{2}}\\
	\ceil{\frac a 2}+1 &= \ceil{\frac {a+2}{2}}
\end{align*}\end{linenomath*}

\subsection{Measures}

In all sections, $(c,d)\in F\delta\times F\delta$, and $F\delta$ has the Haar measure so that $\ofr\delta$ has unit volume. Then $F\delta\times F\delta$ has the product measure. Moreover, $r_1\in E^\times$, and $E^\times$ has the Haar measure so that $\ofr_E^\times$ has unit volume.

\subsection{All Summands Integral}\label{sec:calcAllSummands}

In this section, we evaluate case $(1,0)$  (resp. $(0,1)$) when all summands of $E_6$ and $E_7$ (resp. $E_6^\stc$ and $E_7^\stc$) are integral. We consider two cases: one where $h= V_m-1,\dots,-V_m$ for both integrals, and the other where $h = V_m$ for $\Ocl_\gamma$ and $h = -V_m - 1$ for $\Ocl_{\gamma_\stc}$. In order to make reading this section easier, here are the inequalities that must be satisfied in this case:

\begin{lem}
	Suppose all summands are integral. Then the inequalities defining the set that we must determine the measure of are:
	\begin{linenomath*}\begin{align*}
	\text{For $\Ocl_\gamma$} && \text{For $\Ocl_{\gamma_\stc}$}\\
	h + V_p - 2v(r_1) &\geq 0\\
	h + V_p + v(c) - 2v(r_1) &\geq 0\\
	h + V_p + v(d) - 2v(r_1) &\geq 0 & \text{Same as for $\Ocl_\gamma$}\\
	h + V_p + v(c) + v(d) - 2v(r_1) &\geq 0\\
	V_p - h + 2v(r_1) &\geq 0\\
	V_m + v(c) + h&\geq 0 & V_m + v(c) + h + 1&\geq 0\\
	V_m + v(d) - h &\geq 0 & V_m + v(d) - h -1 &\geq 0
\end{align*}\end{linenomath*}
\end{lem}
\begin{proof}
	We take the terms listed at the beginning of \S\ref{sec:elimination} and set the valuation of each of them to be greater than or equal to zero for $E_1$ to $E_5$ and $E_1^\stc$ to $E_5^\stc$, and set the valuation of each summand to be greater than or equal to zero for $E_6,E_7,E_6^\stc$ and $E_7^\stc$. We note that we have not written down the inequalities for $E_4,E_4^\stc,E_5$, or $E_5^\stc$ because these will automatically be integral given our assumptions on $h$.
\end{proof}

\subsubsection{Case $(1,0)$: The Integral $\Ocl_\gamma$ for $h = V_m$}

Our starting inequalities at the start of \S\ref{sec:calcAllSummands} reduce to the following:
\begin{linenomath*}\begin{align}
		V_m + V_p - 2v(r_1) &\geq 0\\
	V_m + V_p + v(c) - 2v(r_1) &\geq 0\\
	\label{eqn:exred1}V_m + V_p + v(d) - 2v(r_1) &\geq 0\\
	\label{eqn:exred2}V_m + V_p + v(c) + v(d) - 2v(r_1) &\geq 0\\
	V_p - V_m + 2v(r_1) &\geq 0\\
	2V_m + v(c) &\geq 0 \\
	v(d) &\geq 0
\end{align}\end{linenomath*}
We note that since $h=V_m$, the expressions $E_*^\stc$ cannot all be integral since in that case we must have $V_m-1\geq h$. Since $v(d) \geq 0$, we see that \eqref{eqn:exred1} and \eqref{eqn:exred2} are redundant, so we can eliminate them. In subsequent calculations, we shall frequently eliminate the obvious redundant inequalities without note. There are two cases to consider: $v(c) \geq 0$ and $v(c) < 0$.

\case{1}{$v(c)\geq 0$}. The remaining inequalities are
\begin{linenomath*}\begin{align*}
	V_m + V_p\geq 2v(r_1) &\geq V_m-V_p\\
	v(c) &\geq 0\\
	v(d) &\geq 0.
\end{align*}\end{linenomath*}
At this point, the reader may wish to review \S\ref{sec:intArithmetic} containing various identities with floor and ceiling functions. Using these we see that the measure of the corresponding set of solutions is 
\newcommand{\posOne}{
\begin{cases}
	V_p + 1 & \text{if $V_m + V_p$ is even}\\
	V_p & \text{if $V_m + V_p$ is odd}
\end{cases}
}
\begin{linenomath*}\begin{align*}
\posOne	
\end{align*}\end{linenomath*}
\case{2}{$v(c) < 0$}. Now the relevant inequalities are
\begin{linenomath*}\begin{align}
	\label{eqn:posTwo1}V_m + V_p + v(c) \geq 2v(r_1) &\geq V_m-V_p\\
	0 > v(c) &\geq -2V_m \\
	v(d) &\geq 0
\end{align}\end{linenomath*}
We see that \eqref{eqn:posTwo1} implies that $v(c) \geq -2V_p$, which we would have to use instead of $v(c) \geq -2V_m$ if we did not assume $V_p > V_m$. We have the measure
\newcommand{\posTwo}{
(1-q^{-1})\sum_{v(c) = -2V_m}^{-1}q^{-v(c)}\left( \floor{\tfrac{V_m+V_p+v(c)}{2}} - \ceil{\tfrac{V_m-V_p}{2}} + 1 \right)
}
\begin{linenomath*}\begin{align*}
	\posTwo
\end{align*}\end{linenomath*}

\subsubsection{Case $(1,0)$: The Integral $\Ocl_\gamma$: $h=V_m-1,\dots,-V_m$}
The assumptions for $h$ are equivalent to $E_4,E_5\in\ofr$ and $E_4^\stc,E_5^\stc\in \ofr$. Hence, any solution that makes $E_i$ integral will make $E_i^\stc$ integral except when $E_7$ is integral but $E_7^\stc$ is not, which is equivalent to $V_m + v(d) - h \geq 0$ but $V_m + v(d) - h - 1 < 0$. In other words, $v(d) = -V_m + h$. Since $h\leq V_m-1$ by assumption, this implies $v(d) < 0$. Under this additional requirement, we reduce to the following.
\begin{linenomath*}\begin{align*}
	2h + V_p - V_m &\geq 2v(r_1)\\
	2h + V_p - V_m + v(c) &\geq 2v(r_1)\\
	2v(r_1) &\geq h - V_p\\
	v(c) &\geq -h-V_m\\
	v(d) &= h- V_m
\end{align*}\end{linenomath*}
We see again that there are two cases: $v(c) \geq 0$ and $v(c) < 0$.

\case{1}{$v(c)\geq 0$}. Then the inequalities reduce to the product set defined by
\begin{linenomath*}\begin{align*}
	2h + V_p - V_m \geq 2v(r_1)&\geq h - V_p\\
	v(c) &\geq 0\\
	v(d) &= h- V_m
\end{align*}\end{linenomath*}
Hence the measure here is
\newcommand{\posThree}{
q^{V_m-h}(1-q^{-1})\left( \floor{\tfrac{2h+V_p-V_m}{2}} - \ceil{\tfrac{h-V_p}{2}}+1 \right)
}
\begin{linenomath*}\begin{align*}
	\posThree
\end{align*}\end{linenomath*}
\case{2}{$v(c) < 0$}. The relevant inequalities are
\begin{linenomath*}\begin{align*}
	2h + V_p - V_m + v(c) \geq 2v(r_1) &\geq h - V_p\\
	0 > v(c) &\geq -h-V_m\\
	v(d) &= h- V_m
\end{align*}\end{linenomath*}
Hence the measure of the corresponding set is
\newcommand{\posFour}{
q^{V_m-h}(1-q^{-1})^2\sum_{v(c) = -h-V_m}^{-1}q^{-v(c)}\left( \floor{\tfrac{2h+V_p-V_m+v(c)}{2}} - \ceil{\tfrac{h-V_p}{2}} + 1 \right)
}
\begin{linenomath*}\begin{align*}
	\posFour
\end{align*}\end{linenomath*}

\subsubsection{Case $(0,1)$: The Integral $\Ocl_{\gamma_\stc}$: $h=-V_m-1$} Since $h=-V_m-1$, the expressions $E_*$ cannot all be integral. Because $V_m + v(c) + h + 1\geq 0$, putting $h = -V_m - 1$ into this gives $v(c) \geq 0$. We are left with
	\begin{linenomath*}\begin{align*}
	-V_m-1 + V_p - 2v(r_1) &\geq 0\\
	-V_m-1 + V_p + v(d) - 2v(r_1) &\geq 0\\
	V_p +V_m + 1 + 2v(r_1) &\geq 0\\
	v(c) &\geq 0\\
	v(d) &\geq -2V_m
\end{align*}\end{linenomath*}
We do two cases: $v(d) \geq 0$ and $v(d) < 0$.

\case{1}{$v(d) \geq 0$}. We have:
	\begin{linenomath*}\begin{align*}
	 V_p-V_m-1 &\geq 2v(r_1) \geq -V_p - V_m-1\\
	v(c) &\geq 0\\
	v(d) &\geq 0
\end{align*}\end{linenomath*}
The measure of this set is
\newcommand{\negOne}{
\begin{cases}
	V_p & \text{if $V_m + V_p$ is even}\\
	V_p + 1 & \text{if $V_m + V_p$ is odd}
\end{cases}
}
\begin{linenomath*}\begin{align*}
	\negOne
\end{align*}\end{linenomath*}

\case{2}{$v(d) < 0$}.
\begin{linenomath*}\begin{align}
	\label{eqn:vd2vp}V_p - V_m-1 + v(d) \geq 2v(r_1) &\geq -V_p-V_m-1\\
	\label{eqn:drawing1}v(c) &\geq 0\\
	\label{eqn:drawing2}0 > v(d) &\geq -2V_m
\end{align}\end{linenomath*}
Transitivity in~\eqref{eqn:vd2vp} shows that $v(d) \geq -2V_p$, but this is already satisfied under our hypothesis $V_p >  V_m$. We see that the measure of this set is
\newcommand{\negTwo}{
(1-q^{-1})\sum_{v(d) = -2V_m}^{-1}q^{-v(d)}\left( \floor{\tfrac{V_p-V_m-1 + v(d)}{2}}-\ceil{\tfrac{-V_p-V_m-1}{2}}+1 \right)
}
\begin{linenomath*}\begin{align*}
	\negTwo
\end{align*}\end{linenomath*}

\subsubsection{Case $(0,1)$: The Integral $\Ocl_{\gamma_\stc}$: $h=V_m-1,\dots,-V_m$} We just need to evaluate under the conditions that each $E_i^\stc$ is integral but at least one of $E_i$ is not. The only way this can happen is when $v(c) = -V_m-h-1$. In particular, this implies that $v(c) < 0$.

We start with the following (in)equalities:
	\begin{linenomath*}\begin{align}
		\label{eqn:neg3st}h + V_p + v(c) - 2v(r_1) &\geq 0\\
	\label{eqn:negred3}h + V_p + v(c) + v(d) - 2v(r_1) &\geq 0\\
	V_p - h + 2v(r_1) &\geq 0\\
	V_m + v(c) + h + 1&= 0\\
	\label{eqn:negred4}V_m + v(d) - h -1 &\geq 0
\end{align}\end{linenomath*}
We do two cases: $v(d)\geq 0$ and $v(d) < 0$.

\case{1}{$v(d) \geq 0$}. Then, taking the above inequalities, eliminating the redundant ones (\eqref{eqn:negred3} and \eqref{eqn:negred4}), and putting $v(c) = -V_m - h - 1$ gives
\begin{linenomath*}\begin{align*}
	V_p - V_m-1\geq 2v(r_1) &\geq h - V_p\\
	v(d) &\geq 0\\
	v(c) &= -V_m - h - 1.
\end{align*}\end{linenomath*}
The measure of this set is then
\newcommand{\negThree}{
q^{V_m + h + 1}(1-q^{-1})\left( \floor{\tfrac{V_p-V_m-1}{2}} - \ceil{\tfrac{h-V_p}{2}}+1 \right)
}
\begin{linenomath*}\begin{align*}
	\negThree
\end{align*}\end{linenomath*}
\case{2}{$v(d) < 0$}. We again take \eqref{eqn:neg3st}-\eqref{eqn:negred4}, eliminate the redundant \eqref{eqn:neg3st} and make the substitution $v(c) = -V_m - h - 1$ to get:
\begin{linenomath*}\begin{align*}
	V_p - V_m - 1 + v(d) \geq 2v(r_1) &\geq h - V_p\\
	0 > v(d) &\geq h + 1 - V_m\\
	v(c) &= -V_m - h - 1
\end{align*}\end{linenomath*}
Giving us the measure
\newcommand{\negFour}{
	q^{V_m + h + 1}(1-q^{-1})^2\sum_{v(d) = h +1-V_m}^{-1}q^{-v(d)}\left( \floor{\tfrac{V_p-V_m-1 + v(d)}{2}} - \ceil{\tfrac{h-V_p}{2}}+1 \right)
	}
	\begin{linenomath*}\begin{align*}
		\negFour
\end{align*}\end{linenomath*}

\subsection{All Summands Integral: Taking the Difference}

In this section, we take the measure we have found so far for case $(1,0)$ and subtract from it the measure for case $(0,1)$.

\subsubsection{Extreme Cases: $h=V_m$ for $\Ocl_{\gamma}$ and $h=-V_m-1$ for $\Ocl_{\gamma_\stc}$}

Here, we subtract from the measure for $h=V_m$ in $\Ocl_\gamma$ the measure for $h=-V_m-1$ in $\Ocl_{\gamma_\stc}$, for all summands being integral. We find the difference to be:
\begin{linenomath*}\begin{align*}
	&(-1)^{V_m+V_p} + (1-q^{-1})\sum_{j=1}^{2V_m} q^{j}\left( \floor{\tfrac{V_m+V_p-j}{2}} - \ceil{\tfrac{V_m-V_p}{2}} - \floor{\tfrac{V_p-V_m-1-j}{2}}+\ceil{\tfrac{-V_p-V_m-1}{2}} \right)
\end{align*}\end{linenomath*}
In the summation, we see that the terms where $j$ is odd vanish, leaving us with:
\begin{linenomath*}\begin{align*}
&(-1)^{V_m+V_p} + (1-q^{-1})\sum_{j=1}^{V_m} q^{2j}\left( \floor{\tfrac{V_m+V_p-2j}{2}} - \ceil{\tfrac{V_m-V_p}{2}} - \floor{\tfrac{V_p-V_m-1-2j}{2}}+\ceil{\tfrac{-V_p-V_m-1}{2}} \right)
\end{align*}\end{linenomath*}
Simplifying the floor and ceiling functions gives $(-1)^{V_m+V_p}$, so that we get
\begin{linenomath*}\begin{align*}
(-1)^{V_m+V_p}\left(1 + (1-q^{-1})\sum_{j=1}^{V_m}q^{2j}\right)
= (-1)^{V_m+V_p}\left(1 + (q^{2V_m}-1)\frac{q}{q+1}\right).
\end{align*}\end{linenomath*}
\subsubsection{$h = V_m-1,\dots,-V_m$}

There were two cases here: the first, where $v(c) \geq 0$ for $\Ocl_\gamma$ and $v(d) \geq 0$ for $\Ocl_{\gamma_\stc}$, and the second (reverse the inequalities). 

\case{1}{When $v(c) \geq 0$ for $\Ocl_\gamma$ and $v(d)\geq 0$ for $\Ocl_{\gamma_\stc}$.} In this case we had for $\Ocl_\gamma$ the measure:
\begin{linenomath*}\begin{align*}
	\posThree
\end{align*}\end{linenomath*}
and for $\Ocl_{\gamma_\stc}$:
\begin{linenomath*}\begin{align*}
	\negThree
\end{align*}\end{linenomath*}
We have to sum both over $h= -V_m,\dots,V_m-1$, and subtract the second from the first. However, in placing this in the summation, we may replace $h$ in the second expression with $-h-1$, a transformation which preserves the summation range. We do this since then we will get pairs of nicely paired terms in the sum. So we get:
\begin{linenomath*}\begin{align*}
	q^{V_m}(1-q^{-1})\sum_{h=-V_m}^{V_m-1}q^{-h}\left( \floor{\tfrac{2h+V_p-V_m}{2}} -\ceil{\tfrac{V_p-V_m-2}{2}} + \floor{\tfrac{-h-V_p}{2}} - \ceil{\tfrac{h-V_p}{2}}\right)
\end{align*}\end{linenomath*}
where we have converted the floor to ceiling and vice-version in the second equation to make following the computations with \S\ref{sec:intArithmetic} easier. We break the summation into two sums: one over $h= -V_m,-V_m-2,\dots, V_m-2$, and the other over $h = -V_m+1,\dots,V_m-1$. In other words, the first is over integers of the same parity as $V_m$ and the second opposite parity. However, if we look at the opposite-parity case, we see that $2h + V_p - V_m$ and $V_p - V_m-2$ have the same parity, which is the opposite parity of both $-h-V_p$ and $h - V_p$. Hence, the sum vanishes. So we just have sum over $h= -V_m,-V_m+2,\dots,V_m-2$. In this case, we get
\begin{linenomath*}\begin{align*}
	&\phantom{=.}(-1)^{V_m+V_p}q^{V_m}(1-q^{-1})(q^{V_m} + q^{V_m-2} + \cdots + q^{-V_m+2}) \\&= (-1)^{V_m+V_p}q^2(1-q^{-1})(1 + q^2 + \cdots + (q^2)^{V_m-1})\\
	&=(-1)^{V_m+V_p}(q^{2V_m}-1)\frac{q}{q+1}
\end{align*}\end{linenomath*}

\case{2}{When $v(c) < 0$ for $\Ocl_{\gamma}$ and $v(d) < 0$ for $\Ocl_{\gamma_\stc}$} This is a little more lengthy, but not terribly so. We recall the two terms. The first for $\Ocl_\gamma$ is
\begin{linenomath*}\begin{align*}
	\posFour
\end{align*}\end{linenomath*}
The one for $\Ocl_{\gamma_\stc}$ is
\begin{linenomath*}\begin{align*}
	\negFour
\end{align*}\end{linenomath*}
The first thing we do is replace $h$ by $-h-1$ in the $\gamma_\stc$-version, and use $j$ as the index of summation over positive instead of negative numbers. After doing this, converting the appropriate floors to ceilings and vice-versa, \emph{and} subtracting the second from the first, we get:
\begin{linenomath*}\begin{align*}
	q^{V_m-h}(1-q^{-1})^2\sum_{j=1}^{V_m+h}q^j\left( \floor{\tfrac{2h+V_p-V_m-j}{2}} - \ceil{\tfrac{V_p-V_m-2-j}{2}}  +\floor{\tfrac{-h-V_p}{2}} - \ceil{\tfrac{h-V_p}{2}} \right)
\end{align*}\end{linenomath*}
Of course, we have still \emph{not} summed over $h$ yet, and we also note that if $h = -V_m$, the sum is actually empty. Anyways, to make sense of this chaos we write down two separate summations again: one for $h=-V_m,-V_m+2,\dots, V_m-2$ and one for $h=-V_m+1,-V_m+3,\dots, V_m-1$.

\case{2.1}{$h = -V_m,-V_m+2,\dots, V_m-2$.} Here, the upper limit of the summation is even. We also split the summation into two sums, depending on whether $j$ is even or odd:
\begin{linenomath*}\begin{align*}
	q^{V_m-h}(1-q^{-1})^2\Biggl[ \sum_{j=1}^{\tfrac{V_m+h}{2}}q^{2j}\left(   \floor{\tfrac{2h+V_p-V_m-2j}{2}} - \ceil{\tfrac{V_p-V_m-2-2j}{2}}  +\floor{\tfrac{-h-V_p}{2}} - \ceil{\tfrac{h-V_p}{2}}\right) \\
	+ \sum_{j=1}^{\tfrac{V_m+h}{2}}q^{2j-1}\left(  \floor{\tfrac{2h+V_p-V_m-2j+1}{2}} - \ceil{\tfrac{V_p-V_m-1-2j}{2}}  +\floor{\tfrac{-h-V_p}{2}} - \ceil{\tfrac{h-V_p}{2}} \right) \Biggr]
\end{align*}\end{linenomath*}
We again see that the summation where $j$ is odd vanishes, and we simplify the rest to get
\begin{linenomath*}\begin{align*}
	(-1)^{V_m+V_p}q^{V_m-h}(1-q^{-1})^2\sum_{j=1}^{\tfrac{V_m+h}{2}}q^{2j}
	&= (-1)^{V_m+V_p}q^{V_m-h}(1-q^{-1})^2q^2\frac{q^{V_m+h}-1}{q^2-1}\\
	&= (-1)^{V_m+V_p}\frac{q-1}{q+1}(q^{2V_m}-q^{V_m-h})
\end{align*}\end{linenomath*}
As a sanity check, putting in $h=-V_m$ gives zero. Good! Let's sum over $h$ now to get:
\begin{linenomath*}\begin{align}\label{eqn:massive1}\begin{split}
	&\phantom{=.}(-1)^{V_m+V_p}\left[\frac{q-1}{q+1}q^{2V_m}(V_m) - \frac{q-1}{q+1}q^{V_m}(q^{V_m} + q^{V_m-2} + q^{V_m-4} + \cdots + q^{-V_m+2})\right]\\
	&=(-1)^{V_m+V_p}\left[ \frac{q-1}{q+1}q^{2V_m}(V_m) - \frac{q^2}{(q+1)^2}(q^{2V_m}-1)\right]\end{split}
\end{align}\end{linenomath*}

\case{2.2}{$h=-V_m+1,-V_m+3\cdots,V_m-1$.} This time, the upper limit $V_m+h$ is odd. We again split the summation into two sums, depending on whether $j$ is odd or even:
\begin{linenomath*}\begin{align*}
	q^{V_m-h}(1-q^{-1})^2\Biggl[ \sum_{j=1}^{\tfrac{V_m+h-1}{2}}q^{2j}\left(   \floor{\tfrac{2h+V_p-V_m-2j}{2}} - \ceil{\tfrac{V_p-V_m-2-2j}{2}}  +\floor{\tfrac{-h-V_p}{2}} - \ceil{\tfrac{h-V_p}{2}}\right) \\
	+ \sum_{j=1}^{\tfrac{V_m+h+1}{2}}q^{2j-1}\left(  \floor{\tfrac{2h+V_p-V_m-2j+1}{2}} - \ceil{\tfrac{V_p-V_m-1-2j}{2}}  +\floor{\tfrac{-h-V_p}{2}} - \ceil{\tfrac{h-V_p}{2}} \right) \Biggr]
\end{align*}\end{linenomath*}
This time, the opposite happens: in other words, the summation with even powers of $q$ vanishes, and we are left with:
\begin{linenomath*}\begin{align*}
	\phantom{=.}-(-1)^{V_m+V_p}q^{V_m-h}(1-q^{-1})^2\sum_{j=1}^{\tfrac{V_m+h+1}{2}}q^{2j-1} 
	&=-(-1)^{V_m+V_p}\frac{q-1}{q+1}(q^{2V_m} - q^{V_m-h-1})
\end{align*}\end{linenomath*}
Summing over $h = -V_m+1,-V_m+3,\cdots,V_m-1$ gives
\begin{linenomath*}\begin{align*}
	-(-1)^{V_m+V_p}\left[ \frac{q-1}{q+1}q^{2V_m}(V_m) -\frac{1}{(q+1)^2}(q^{2V_m}-1)\right]	
\end{align*}\end{linenomath*}
We add this to \eqref{eqn:massive1} to get
\begin{linenomath*}\begin{align*}
	(-1)^{V_m + V_p}\left[ \frac{1}{(q+1)^2}(q^{2V_m}-1) - \frac{q^2}{(q+1)^2}(q^{2V_m}-1) \right] = (q^{2V_m}-1)\frac{1-q}{1+q}
\end{align*}\end{linenomath*}
\subsubsection{Gathering of Terms}\label{sec:wheretheansweris}
We have now collected all the terms in our integral for the ``all summands positive'' case. We add them together:
\begin{linenomath*}\begin{align*}
	(-1)^{V_m+V_p}\left[ 1 + (q^{2V_m}-1)\frac{q}{q+1} + (q^{2V_m} - 1)\frac{q}{q+1} + (q^{2V_m}-1)\frac{1-q}{q+1} \right] &= (-1)^{V_m+V_p}q^{2V_m}.
\end{align*}\end{linenomath*}

\subsection{Summands of $E_6$ Integral Only}
The last case is the case of the summands of $E_6$ being integral only. This case is a little different because here, it will be impossible that both $E_7$ and $E_7^\stc$ will be simultaneously satisfied. 

\begin{lem}
	For the summands of $E_6=E_6^\stc$ to be integral and the summands of $E_7$ (resp. $E_7^\stc$) to be not integral, the following inequalities have to be satisfied:
\begin{linenomath*}\begin{align*}
	\text{For $\Ocl_\gamma$} & & \text{(resp. For $\Ocl_{\gamma_\stc}$ )}\\
		h + V_p - 2v(r_1) &\geq 0\\
	h + V_p + v(c) - 2v(r_1) &\geq 0\\
	h + V_p + v(d) - 2v(r_1) &\geq 0\\
	h + V_p + v(c) + v(d) - 2v(r_1) &\geq 0\\
	V_p - h + 2v(r_1) &\geq 0\\
	V_m + v(c) + h &< 0 & V_m + v(c) + h + 1 &< 0\\
	V_m + v(d) - h &< 0 & V_m + v(d) -h - 1 &< 0\\
\end{align*}\end{linenomath*}
\end{lem}
\begin{proof}
	We set the valuations of the expressions at the beginning of \S\ref{sec:elimination} to be greater than or equal to zero for $E_1$ to $E_5$, and we do the same for the summands  of $E_6$. We also set the summands of $E_7$ (resp. $E_7^\stc$) to have valuation less than zero. As usual, we have omitted the inequalities for $E_4,E_4^\stc,E_5,E_5^\stc$  since these are equivalent to our assumptions on $h$.
\end{proof}

\subsubsection{Case $(1,0)$: The Integral $\Ocl_\gamma$}\label{sec:e6onlypos}
Here, we come up with an expression for $h = -V_m,\dots,V_m$. We make three straightforward observations:
\begin{enumerate}
	\item Since $V_m + h + v(c) < 0$ and $V_m + h \geq 0$, we must have $v(c) < 0$, and similarly, $V_m - h + v(d) < 0$ implies that $v(d) < 0$
	\item Since the summands of $E_7$ are not integral, we must have the valuations of these terms equal. Hence $v(d) -2h=  v(c)$.
	\item In addition, since we are solving a congruence in $E_7$, once $v(d)$ is chosen, the measure of $\{ c : v(E_7)\geq 0\}$ is $q^{V_m+h}$. 
\end{enumerate}
Making the substitution $v(c) = v(d) - 2h$ and eliminating redundancies gives the inequalities
\begin{linenomath*}\begin{align*}
	V_p - h + 2v(d) \geq 2v(r_1)&\geq h - V_p \\
	 h - V_m &> v(d)
\end{align*}\end{linenomath*}
We see from the first that $v(d) \geq h - V_p$, which gives a lower limit for $v(d)$. Thus, the measure of this solution set is
\newcommand{\posFive}{
q^{V_m+h}(1-q^{-1})\sum_{v(d) = h-V_p}^{h-V_m-1}q^{-v(d)}\left( \floor{\tfrac{V_p-h+2v(d)}{2}} - \ceil{\tfrac{h-V_p}{2}}+1 \right).
}
\begin{linenomath*}\begin{align}\label{eqn:posFive}
\posFive	
\end{align}\end{linenomath*}

\subsubsection{Case $(0,1)$: The Integral $\Ocl_{\gamma_\stc}$} Here we have essentially the same three observations as in \S\ref{sec:e6onlypos}, suitably modified.
\begin{enumerate}
	\item Since $V_m + h + v(c) +1< 0$ and $V_m + h+1 \geq 0$, we must have $v(c) < 0$, and similarly, $V_m - h + v(d)-1 < 0$ implies that $v(d) < 0$
	\item Since the summands of $E_7^\stc$ are not integral, we must have the valuations of these terms equal. Hence $v(d) -2h-2=  v(c)$.
	\item In addition, since we are solving a congruence in $E_7^\stc$, once $v(d)$ is chosen, the measure of $\{ c : v(E_7)\geq 0\}$ is $q^{V_m+h+1}$. 
\end{enumerate}
We have the inequalities:
\begin{linenomath*}\begin{align*}
	V_p-h+2v(d)-2 \geq 2v(r_1) &\geq h-V_p\\
	 h + 1 - V_m > v(d) &\geq h  -V_p+1
\end{align*}\end{linenomath*}
where the lower limit for $v(d)$ comes from the first inequality. We obtain the measure
\begin{linenomath*}\begin{align}\label{eqn:negFive}
	q^{V_m+h+1}(1-q^{-1})\sum_{v(d) = h - V_p+1}^{h - V_m}q^{-v(d)}\left( \floor{\tfrac{V_p -h+2v(d)-2}{2}} - \ceil{\tfrac{h - V_p}{2}}+1 \right)
\end{align}\end{linenomath*}

\subsubsection{Gathering it Together} In this case, we see that after shifting the index of summation in \eqref{eqn:negFive} so that $v(d)$ starts at $h - V_p$, we get exactly \eqref{eqn:posFive}, so that the two cancel.

\section{Results and Interpretations}\label{sec:results} The calculations of \S\ref{sec:calculations}, shown particularly in \S\ref{sec:wheretheansweris} give
\begin{linenomath*}\begin{align*}
	\Ocl_\gamma^\kappa(\mathbf{1}) = (-1)^{V_m+V_p}q^{2V_m}
\end{align*}\end{linenomath*}
We recall that $\gamma=\diag(x,y,-y,-x)$ and $V_m = v(x-y)$. We have suggested that the corresponding endoscopic space $(H,\theta_H)$ is two copies of $\UU(1)\times\UU(1)\hookrightarrow \UU(2)$, given as follows. Set $J_2 = \left( \begin{smallmatrix} 0 & 1\\1 & 0\end{smallmatrix} \right)$. We have $\UU(2) := \{ g\in \GL_{2} : J\overline{g}^{-t}J = g\}$ and $\UU_2^{J_2}\cong \UU(1)\times \UU(1)$, so that $\theta_H = (J_2, J_2)$. A trivial computation shows that for $\gamma_H = (\diag(x,-x),\diag(y,-y))$, the corresponding stable orbital integral is just one. Although this does not suggest a way to define endoscopic pairs $(H,\theta_H)$ in general, it is the likely choice given the situation with the adjoint case. Hence:
\begin{thm}For $(\UU_4,\theta)$ and $(\UU_2\times\UU_2,\theta_H)$ with $\gamma=\diag(x,y,-y,-x)$ nearly singular and $\gamma_H = \diag(x,-x)\times\diag(y,-y)$, we have the identity
	\begin{linenomath*}
		\begin{align*}
			\Ocl^\kappa_\gamma(\mathbf{1}_{\gfr_1(\ofr)}) = (-1)^{V_m+V_p}q^{2V_m}\Scl\Ocl_{\gamma_H}(\mathbf{1}_{\hfr_1(\ofr)}).
		\end{align*}
	\end{linenomath*}
\end{thm}

The factor of $(-1)^{V_m+V_p}$ is not terribly mysterious, and one could likely eliminate it by using the relative Kostant-Weierstrass section shown to exist in~\cite{levyInvolutions} and following the ideas in \cite{kotttrans}, so we concentrate on the power of $q$. We offer the following tentative interpretation of the power of $q$.

We define $\widetilde{\gamma}\in\ufr^\theta(F)$ by $\lambda\diag(x,y,y,x)$ where $\lambda\in E$ is such that $v(\lambda) = 0$ and $\overline{\lambda} = -\lambda$. The map $\gamma\mapsto \widetilde{\gamma}$ gives an $F$-linear isomorphism between $\Lie(I_\gamma)(F)$ and $\afr(F)$. But centraliser $I_\gamma$ of $\gamma$ in $G_0$ is $I_\gamma = \{ \diag(a_{11},a_{22},a_{22},a_{11}) : a_{ii}\overline{a}_{ii} = 1\}$. Its roots, or nonzero weights, of its action on $\gfr_1$ (and on $\gfr_0$), are given (in terms of homomorphisms to $\Res{E/F}(\G_m)$, using adjointness of the restriction of scalars) by
\begin{linenomath*}\begin{align*}
\diag(a_{11},a_{22},a_{22},a_{11})&\mapsto a_{11}a_{22}^{-1},\\
\diag(a_{11},a_{22},a_{22},a_{11})&\mapsto a_{11}^{-1}a_{22}.
\end{align*}\end{linenomath*}
Each root space being two-dimensional. Let $D : \Lie(I_\gamma)\to F\delta$ be the discriminant function $\prod_{\alpha} (d\alpha)^{r_\alpha}$ where $r_\alpha$ is the $F$-dimension of the corresponding root space. Na\"ively using the formula in~\cite{ngo2010lemme} with this discriminant function on $\widetilde{\gamma}$ gives
\begin{linenomath*}\begin{align*}
	D(\widetilde{\gamma})/2 &= \frac{2v(x-y) + 2v(x-y)}{2}\\
	&= 2V_m.
\end{align*}\end{linenomath*}
We shall attempt a reasonable explanation in a future work.

\bibliographystyle{alpha}
\bibliography{polakmain}

\end{document}